\documentclass[11pt]{amsart}

\usepackage{latexsym, amssymb, amsmath, amscd, amsfonts}
\usepackage{mathrsfs}

\usepackage{color,graphicx}
\usepackage{tikz}
\usetikzlibrary{positioning,calc}

\numberwithin{equation}{section}
\newtheorem{theorem}{Theorem}[section]
\newtheorem{lemma}[theorem]{Lemma}
\newtheorem{proposition}[theorem]{Proposition}

\newtheorem{definition}[theorem]{Definition}
\newtheorem{example}[theorem]{Example}
\newtheorem{remark}[theorem]{Remark}
\newtheorem{notation}[theorem]{Notation}

\begin{document}

\title[Laurent polynomials]{A filtration on the ring of Laurent polynomials
and representations of the general linear Lie algebra}

\author{Cheonho Choi}
\address{Department of Mathematics, Korea University,
 145 Anam-ro Seongbuk-gu, Seoul 02841, South Korea}
\email{cheonho@korea.ac.kr}

\author{Sangjib Kim}
\address{Department of Mathematics, Korea University,
 145 Anam-ro Seongbuk-gu, Seoul 02841, South Korea}
\email{sk23@korea.ac.kr}
 
\author{HaeYun Seo}
\address{Department of Mathematics, University of Maryland, 
William E. Kirwan Hall, 4176 Campus Drive, College Park, MD 20742-4015, USA}
\email{hys2500@terpmail.umd.edu}
 

\begin{abstract}
We first present a filtration on the ring $L_n$ of Laurent polynomials 
such that the direct sum decomposition of its associated graded ring $gr L_n$ 
agrees with the direct sum decomposition of $gr L_n$, as a module over 
the complex general linear Lie algebra $\mathfrak{gl}(n)$,
into its simple submodules. 
Next, generalizing the simple modules occurring in the associated graded ring $gr L_n$, 
we give some explicit constructions of weight multiplicity-free
irreducible representations of $\mathfrak{gl}(n)$.
\end{abstract}


\maketitle


\section{Introduction}


In this section, we  give a brief summary of our results.

\subsection{The ring of polynomials}

The ring $P_n = \mathbb{C}[x_1, ..., x_n]$ of polynomials in $n$ indeterminates 
over the complex numbers $\mathbb{C}$ is a $\mathbb{Z}$-graded algebra
\begin{equation}\label{poly-decomp}
P_n = \bigoplus_{m \in \mathbb{Z}} P^{(m)}_n
\end{equation}
where $P^{(m)}_n$ is the space of homogeneous polynomials of degree $m$. 
As a vector space, $P_n$ becomes a module over
the complex general linear Lie algebra $\mathfrak{gl}(n)=\mathfrak{gl}_n(\mathbb{C})$ 
under the action
\begin{equation}\label{poly-action-Pn}
A \cdot f = \sum_{ij} a_{ij} x_i \frac{\partial f}{\partial x_j}
\quad
\text{for $A=(a_{ij}) \in \mathfrak{gl}(n)$ and $f \in P_n$}. 
\end{equation}
Then, the direct sum decomposition $\eqref{poly-decomp}$ of $P_n$ 
as a graded ring agrees with the decomposition of $P_n$ as 
a $\mathfrak{gl}(n)$-module into its simple submodules $P^{(m)}_n$.
They are the finite dimensional 
representations of $\mathfrak{gl}(n)$ labeled by Young diagrams 
with single rows.

\subsection{The ring of Laurent polynomials}

The first goal of this paper is to obtain 
an analogous result of the above observation for the ring of Laurent polynomials
\[
L_n = \mathbb{C}[x_1^{\pm 1}, x_2^{\pm 1}, ..., x_n^{\pm 1}]. 
\]
It turns out that a filtration and its associated graded structure 
give us an answer. Note that \eqref{poly-decomp} can be seen 
as the graded ring associated with the $\mathbb{Z}$-filtration 
of $P_n$ given by degree.

We will define a filtration on $L_n$ by a partially ordered monoid constructed from 
integers and subsets of $\{1, 2,..., n\}$
\[
L_n = \bigcup_{(m,J)\in \mathbb{Z} \times \mathscr{P}_n} L_n^{\leqslant (m, J)}
\]
and show that the direct sum decomposition of its associated graded ring 
\[
gr L_n = \bigoplus_{(m,J)\in \mathbb{Z} \times \mathscr{P}_n} 
           { L_n^{\leqslant (m, J)} / L_n^{< (m, J)} }
\]
provides the decomposition of $gr L_n$, as a $\mathfrak{gl}(n)$-module, 
into its simple submodules.

Extending the space with the action \eqref{poly-action-Pn} of $\mathfrak{gl}(n)$ 
from $P_n$ to $L_n$, we identify Laurent monomials 
$\mathbf{x}^{\mathbf{k}}=x_1^{k_1} x_2^{k_2} \cdots x_n^{k_n}$ 
with integral points $\mathbf{k}=(k_1, k_2 ..., k_n)$ in $\mathbb{R}^n$. 
Note that they are weight vectors with respect to the Cartan subalgebra of 
$\mathfrak{gl}(n)$ consisting of diagonal matrices.
Since this action preserves the degree of monomials, 
we can  focus on integral points on the hyperplane 
$k_1 + \cdots + k_n =m$ for each $m \in \mathbb{Z}$. 

One of main difficulties in studying the $\mathfrak{gl}(n)$-module structure of 
$L_n$ is that the symmetric behavior of raising and 
lowering operators we had when working with $P_n$ is not trivial anymore. 
For example, when $n=2$ as in Figure \ref{figure-01},
\[
\begin{bmatrix} 0 & 0 \\ 1 & 0 \end{bmatrix} \cdot \, 
	x_1^{k_1} x_2^{m - k_1} = k_1 x_1^{k_1 -1} x_2^{m+1- k_1} \text{\ \ and \ }
\begin{bmatrix} 0 & 1 \\ 0 & 0 \end{bmatrix} \cdot \, 
	x_1^{m- k_2} x_2^{k_2}  = k_2 x_1^{m+1 - k_2} x_2^{k_2 -1}. 
\]
The cases $k_1=0$ and $k_2=0$ divide the line $k_1 + k_2 =m$ into 
three parts.
The monomials with $k_1 \geq 0$ and $k_2 \geq 0$ can be obtained 
by applying some elements of $\mathfrak{gl}(2)$ to monomials 
with $k_1 k_2 < 0$. However, monomials with $k_1 k_2 < 0$ 
cannot be obtained from the ones with $k_1 \geq 0$ and $k_2 \geq 0$.

\begin{figure}
\centering
\begin{tikzpicture}[scale=1.1] 
\draw[->] (-3,0) -- (4,0); \foreach \x in {-3,-2,-1, 1 ,2,3}
\draw[shift={(\x,0)}] (0pt,2pt) -- (0pt,-2pt) node[gray, below] {\footnotesize $\x$};
\draw[->] (0,-3) -- (0,4); \foreach \y in {-3,-2,-1, 1 , 2 ,3}
\draw[shift={(0,\y)}] (2pt,0pt) -- (-2pt,0pt) node[gray, left] {\footnotesize $\y$};
\node[gray] at (0.15,-0.3) {\footnotesize $0$};
\node at (4.1,-0.3) {\footnotesize $k_1$};
\node at (0.3,3.7) {\footnotesize $k_2$};
\node[circle, fill=black, inner sep=0pt, minimum size=3.5pt] (p0) at (-2,3) {};
\node[circle, fill=black, inner sep=0pt, minimum size=3.5pt] (p1) at (-1,2) {};
\node[circle, fill=black, inner sep=0pt, minimum size=3.5pt] (p2) at (0,1) {};
\node[circle, fill=black, inner sep=0pt, minimum size=3.5pt] (p3) at (1,0) {};
\node[circle, fill=black, inner sep=0pt, minimum size=3.5pt] (p4) at (2,-1) {};
\node[circle, fill=black, inner sep=0pt, minimum size=3.5pt] (p5) at (3,-2) {};
\node (q0) at (5,-2) {};
\node (q2) at (5,-2.5) {};
\draw (-2.2, 3.2) -- (3.2, -2.2);             		 
\draw[red, ->, line width=0.4mm] (-2,3.3) -- ({0.1*sqrt(2) +1}, {-0.1*sqrt(2) + 0.3});
\draw[red, ->, line width=0.4mm] ({0.1*sqrt(2) +2},{-0.1*sqrt(2) -0.7}) -- ({0.1*sqrt(2)+3}, {-0.1*sqrt(2) -1.7});
\draw[red, ->, line width=0.4mm] ({-0.1*sqrt(2) -1},{0.1*sqrt(2)+1.7}) -- ({-0.1*sqrt(2)-2}, {0.1*sqrt(2)+2.7});                		 
\draw[red, ->, line width=0.4mm] (3,-2.3) -- ({-0.1*sqrt(2)}, {0.1*sqrt(2) +0.7});  
\end{tikzpicture}
\caption{The action of $\mathfrak{gl}(2)$ on $x_1^{k_1} x_2^{k_2}$ with $k_1 + k_2 = 1$.}
\label{figure-01}
\end{figure}
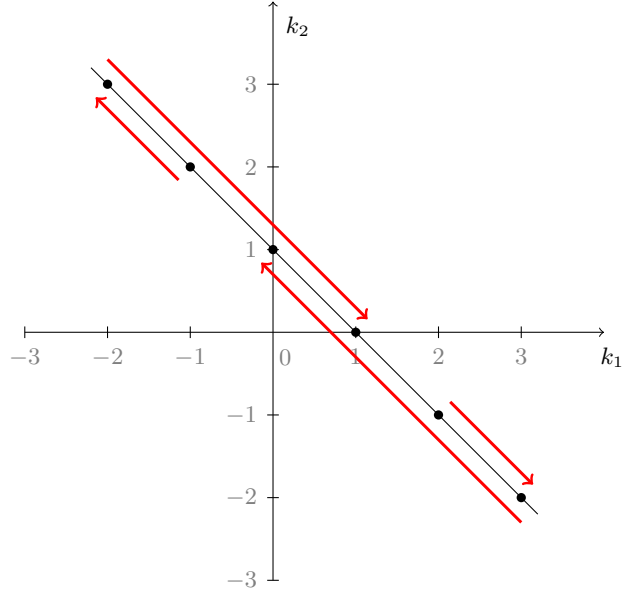

More generally, the planes $k_j=0$ divide the hyperplane $k_1 + k_2 + \cdots + k_n =m$
into regions labeled by the signs of the coordinates $k_i$. 
Then, for each $i$, we can obtain weight vectors 
$\mathbf{x}^{\mathbf{k}}$ with $k_i \geq 0$ starting  
from the ones with $k_i <0$ by successively applying some elements of $\mathfrak{gl}(n)$, 
but the opposite way is not possible. See Figure \ref{figure02}.

Therefore, our indecomposable submodules in $L_n$ and simple modules 
obtained from their quotients are labeled by degree $m$ of 
$\mathbf{x}^{\mathbf{k}}$ and subsets $J$ of $\{1,2,...,n\}$ indicating
the position of possible negative components in $\mathbf{k}$. 
Their structures depend heavily on $m$ and the cardinality of $J$. 
We will give a clear case-by-case analysis of them.

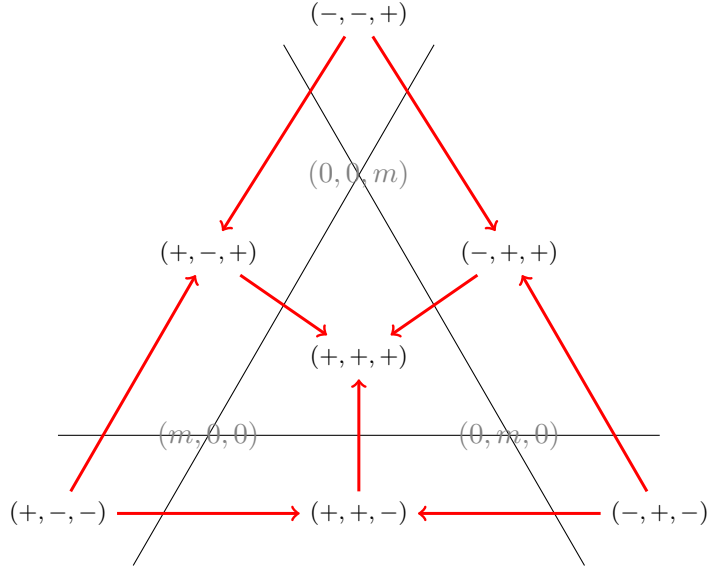
\begin{figure}
\centering
\begin{tikzpicture}[scale=0.4]
\draw (-10,0) -- (10,0) ;
\draw (-7.5,{-0.5*5*sqrt(3)}) -- (2.5,{1.5*5*sqrt(3)});
\draw (7.5, {-0.5*5*sqrt(3)}) -- (-2.5,{1.5*5*sqrt(3)});
\node[color=gray] (x)  at (-5,0) {$(m,0,0)$};
\node[color=gray] (y)  at (5,0) {$(0,m,0)$};
\node[color=gray] (z)  at (0,{5*sqrt(3)}) {$(0,0,m)$};
\node (p0) at (0,{0.3*5*sqrt(3)}) {\footnotesize $(+, +, +)$};
\node (p1) at (0,{-0.3*5*sqrt(3)}) {\footnotesize $(+,+,-)$};
\node (p2) at (-5,{0.7*5*sqrt(3)}) {\footnotesize $(+, -, +)$};
\node (p3) at (5,{0.7*5*sqrt(3)}) {\footnotesize $(-, +,+)$};
\node (p4) at (-10,{-0.3*5*sqrt(3)}) {\footnotesize $(+, -, -)$};
\node (p5) at (10,{-0.3*5*sqrt(3)}) {\footnotesize $(-, +, -)$};
\node (p6) at (0,14) {\footnotesize $(-,-,+)$};
\draw[->, line width=0.4mm] (p1) edge [color=red] (p0);
\draw[->, line width=0.4mm] (p2) edge [color=red] (p0);
\draw[->, line width=0.4mm] (p3) edge [color=red] (p0);
\draw[->, line width=0.4mm] (p4) edge [color=red] (p1);
\draw[->, line width=0.4mm] (p4) edge [color=red] (p2);
\draw[->, line width=0.4mm] (p5) edge [color=red] (p1);
\draw[->, line width=0.4mm] (p5) edge [color=red] (p3);
\draw[->, line width=0.4mm] (p6) edge [color=red] (p2);
\draw[->, line width=0.4mm] (p6) edge [color=red] (p3);
\end{tikzpicture}
\caption{The action of $\mathfrak{gl}(3)$ on $x_1^{k_1} x_2^{k_2} x_3^{k_3}$ with $k_1 + k_2 + k_3 =m$ ($m >0$).}
\label{figure02}
\end{figure}

\subsection{Representations of the general linear Lie algebra}

Our next goal is  
to provide explicit constructions of weight multiplicity-free irreducible 
representations of $\mathfrak{gl}(n)$ obtained by twisting the action \eqref{poly-action-Pn}.
For a general theory on weight multiplicity-free representations 
of simple Lie algebras, see \cite{BL87} and references therein.

Motivated by works on  weight modules of the Lie algebra 
of diffeomorphisms of the $n$-dimensional torus (see, for example, \cite{Rao94, Rao96, GZ11}),
for each $\boldsymbol{\alpha}=(\alpha_1, ..., \alpha_n) \in \mathbb{C}^n$, 
we will define a representation  $L^{\boldsymbol{\alpha}}_n$ of $\mathfrak{gl}(n)$ 
on the vector space $L_n$ (see Definition \ref{alpha-action-deff}).
Then, we investigate two families of its submodules, 
$L^{\boldsymbol{\alpha}}_n(m,j)$ and $V^{\boldsymbol{\alpha}}_n(m,J)$, 
parameterized by integers $m$, $j$, and subsets $J$ of $\{i: \alpha_i=0 \}$. 
We can obtain explicit simple $\mathfrak{gl}(n)$-modules from 
the decomposition of the quotient modules 
\[
L^{\boldsymbol{\alpha}}_n(m,j) /\, L^{\boldsymbol{\alpha}}_n(m,j-1)
  = \bigoplus_{J: |J|=j} W^{\boldsymbol{\alpha}}_n(m,J)
\]
where $W^{\boldsymbol{\alpha}}_{n}(m,J)$ are simple modules defined by
\[
W^{\boldsymbol{\alpha}}_{n}(m,J) = \left( V^{\boldsymbol{\alpha}}_{n}(m,J) 
  +L^{\boldsymbol{\alpha}}_{n}(m,j-1) \right) / L^{\boldsymbol{\alpha}}_{n}(m,j-1).
\]

Among these simple modules, 
there are highest weight modules with highest weights of the form 
$\psi^{\lambda} \in \mathfrak{h}^{\ast}$ where 
$\lambda=(-1, ..., -1, z, 0,..., 0) \in \mathbb{C}^n$ 
including the finite dimensional ones having  
integral dominant weights with $\lambda=(k,0,...,0)$ and $(-1,..., -1, \ell)$ 
for $k \geq 0$ and $\ell \leq -1$.

\medskip


\section{A filtration on $L_{n}$ and simple modules in $gr L_n$}


In this section, we impose a filtration on the ring  
\[
L_n = \mathbb{C}[x_1^{\pm 1}, ..., x_n^{\pm 1}]
\]
of Laurent polynomials in $n$ indeterminates over the complex numbers $\mathbb{C}$, 
and then show that the graded structure of its associated graded ring 
is compatible with the module structure of $L_n$ over the complex general linear Lie 
algebra $\mathfrak{gl}(n)=\mathfrak{gl}_n(\mathbb{C})$. 

Recall that $\mathfrak{gl}(n)$ is the Lie algebra of $n \times n$ complex 
matrices with the usual matrix addition and 
the Lie bracket given by the commutator of two matrices.
We will write  $\mathcal{U}_n=\mathcal{U}(\mathfrak{gl}(n))$ 
for the universal enveloping algebra of $\mathfrak{gl}(n)$.

\subsection{Submodules of $L_n$}

The complex vector space $L_n$ is spanned by monomials 
$\mathbf{x}^{\mathbf{k}}= x^{k_1}_1 x^{k_2}_2 \cdots x^{k_n}_n$ 
for $\mathbf{k}=(k_1, k_2,...,k_n) \in \mathbb{Z}^n$. 
We define some subspaces of $L_n$.

\begin{definition}
Let $m$ be an integer, $j$ be an integer with $0 \leq j \leq n$, and $J$ be 
a subset of $\{ 1,2,...,n\}$.
\begin{enumerate}
\item Let $V_{n}(m,J)$ be the subspace of $L_n$ spanned by all the monomials 
$x^{k_1}_1 x^{k_2}_2 \cdots x^{k_n}_n$ 
such that 
\[
\sum_{i=1}^n k_i = m  \quad \text{and} \quad \{i : k_i <0 \} \subseteq J.
\]

\item Let $L_{n}(m,j)$ be the sum of the subspaces $V_n(m, J)$ of $L_n$
\[
L_{n}(m,j) = \sum_{J: |J|=j} V_{n}(m,J)
\]
over all subsets $J$ of $\{1,2,...,n\}$ having $j$ elements.
\end{enumerate}
\end{definition}

It follows directly from the definition that
\[
V_n(m, J_1) \subseteq V_n(m ,J_2) \quad \text{for all $J_1 \subseteq J_2$.}
\]
Then, we have
\[
L_n (m,j-1) \subseteq L_n (m,j) 
\]
and their  quotient can be expressed as
\begin{align*}
L_{n}(m,j) / L_{n}(m,j-1)
 &= \left( \sum_{J: |J|=j} V_{n}(m,J) \right) / L_{n}(m,j-1) \\
 &= \sum_{J: |J|=j} \left( V_{n}(m,J) 
  +L_{n}(m,j-1) \right) / L_{n}(m,j-1).
\end{align*}

\begin{definition}
For an integer $m$ and a subset $J$ of $\{ 1,2,...,n\}$ having $j$ elements, 
we let $W_{n}(m,J)$ denote the subspace 
\[
W_{n}(m,J) 
= \left( V_{n}(m,J) 
  +L_{n}(m,j-1) \right) / L_{n}(m,j-1)
\]
of the quotient space $L_{n}(m,j) / L_{n}(m,j-1)$.
\end{definition}

The spaces $V_n(m,J)$, $L_n(m,j)$, and $W_n(m,J)$ are special cases of 
the ones defined in Definition \ref{L-V-def} and Definition \ref{def-W} 
with $\boldsymbol{\alpha}=\mathbf{0}$ and $I_{\boldsymbol{\alpha}}=\{1,2,...,n \}$, and 
they are  modules over $\mathcal{U}_n$ with respect to
\[
A \cdot f = \sum_{ij} a_{ij} x_i \frac{\partial f}{\partial x_j}
\]
for $A=(a_{ij}) \in \mathfrak{gl}(n)$ and $f \in L_n$. See 
Theorems \ref{V-cyclic-plus}, \ref{V-cyclic-minus}, \ref{L-indecomp-positive}, 
and \ref{L-indecomp-negative}. 
Moreover, $W_{n}(m,J)$ are simple modules. See Theorem \ref{W-subquot}.

\begin{lemma}\label{W-quot-0-isom}
For $m \in \mathbb{Z}$ and a subset $J$ of $\{ 1, 2, ..., n\}$ having $j$ elements, 
as a $\mathcal{U}_n$-module,
\[
W_{n}(m,J)  \; \cong \; V_{n}(m,J) /  \sum_{J'} V_{n}(m,J')
\]
where the summation is over all subsets $J'$ of $J$ having $j-1$ elements
\end{lemma}
\begin{proof}
Note that
\begin{align*}
W_{n}(m,J) & \; = \;   \left( V_{n}(m,J) +L_{n}(m,j-1) \right) / L_{n}(m,j-1)  \\ 
                    & \; \cong \; V_{n}(m,J) / \left( V_{n}(m,J) \cap L_{n}(m,j-1) \right) .
\end{align*}
Then, the statement follows from the following observation
\[
V_{n}(m,J) \cap L_{n}(m,j-1) = \sum_{J'} V_{n}(m,J')
\]
where the sum is over all subsets $J'$ of $J$ having $j-1$ elements.
\end{proof}

\subsection{Filtration by a partially ordered monoid}

Let $\mathscr{P}_n$ be the set of all subsets of $\{1,2,...,n\}$. 
On the set $\mathbb{Z} \times \mathscr{P}_n$, 
we define the partial order
\[
(m_1, J_1) \leqslant (m_2, J_2)
\]
if $m_1 \leq m_2$ and $J_1 \subseteq J_2$, and the multiplication
\[
(m_1, J_1) \ast (m_2, J_2) = (m_1 + m_2, J_1 \cup J_2).
\]

With $\leqslant$ and $\ast$, 
$\mathbb{Z} \times \mathscr{P}_n$ becomes 
a partially ordered monoid with the identity $(0, \varnothing)$, 
and using this monoid we want to impose a filtration on $L_n$. 
For basic properties of a filtration of a ring given by a partially ordered monoid, 
we refer to \cite[\S I.12]{BG02}.

\begin{definition}
For each $(m, J) \in \mathbb{Z} \times \mathscr{P}_n$, we define
\[
L_n^{\leqslant (m, J)}  = \sum_{(m_1,J_1)} V_n (m_1, J_1) \quad \text{and} \quad
L_n^{< (m, J)}  = \sum_{(m_2, J_2)} V_n (m_2, J_2)
\]
where the first summation is over all $(m_1, J_1)$ such that $(m_1, J_1) \leqslant (m, J)$; 
and the second summation is over all $(m_2, J_2)$  such that $(m_2, J_2) \leqslant (m, J)$ 
but $(m_2, J_2) \ne (m, J)$.
\end{definition}

\begin{proposition}
The family $\{ L_n^{\leqslant (m, J)} : (m, J) \in \mathbb{Z} \times \mathscr{P}_n \}$ 
of subspace of $L_n$ defines a filtration on $L_n$ by the partially ordered monoid $\mathbb{Z} \times \mathscr{P}_n$.
\end{proposition}
\begin{proof}
We need to check the following conditions (see \cite[\S I.12]{BG02}).
\begin{enumerate}
\item $1 \in L_n^{\leqslant (0, \varnothing)}$,
\item for all $(m_1, J_1) \leqslant (m_2, J_2)$,
\[
L_n^{\leqslant (m_1, J_1)} \subseteq L_n^{\leqslant (m_2, J_2)},
\] 
\item for all $(m_1, J_1)$ and $(m_2, J_2)$,
\[
L_n^{\leqslant (m_1, J_1)} L_n^{\leqslant (m_2, J_2)} \subseteq L_n^{\leqslant (m_1, J_1) \ast (m_2, J_2)},
\] 
\item the union of all such subspaces is equal to $L_n$
\[
\bigcup_{(m, J) \in \mathbb{Z} \times \mathscr{P}_n} L_n^{\leqslant (m, J)} = L_n.
\]
\end{enumerate}
All of them follow directly from the definitions of $L_n^{\leqslant (m, J)}$ and $V_n(m,J)$.
\end{proof}

For $(m_1, J_1)$ and $(m_2, J_2)$ in $\mathbb{Z} \times \mathscr{P}_n$, we have 
\[
L_n^{\leqslant (m_1, J_1)}  L_n^{< (m_2, J_2)} + 
L_n^{< (m_1, J_1)}  L_n^{\leqslant (m_2, J_2)}
\; \subseteq \; L_n^{< (m_1, J_1) \ast (m_2, J_2)}
\]  
and therefore the quotient spaces 
$L_n^{\leqslant (m, J)} / L_n^{< (m, J)}$  
form the homogeneous components of the associated graded ring of $L_n$.


\begin{theorem}
With the filtration $\{ L_n^{\leqslant (m, J)} : (m, J) \in \mathbb{Z} \times \mathscr{P}_n \}$ 
imposed on $L_n$, the direct sum decomposition of the associated graded ring 
\[
gr L_n = \bigoplus_{(m,J)\in \mathbb{Z} \times \mathscr{P}_n} 
           { L_n^{\leqslant (m, J)} / L_n^{< (m, J)} }
\]
agrees with the decomposition of $gr L_n$ into its simple submodules over $\mathcal{U}_n$. 
In particular, for each $(m, J) \in \mathbb{Z} \times \mathscr{P}_n$, as a $\mathcal{U}_n$-module,
\[
L_n^{\leqslant (m, J)} / L_n^{< (m, J)} \cong W_n(m, J).
\]
\end{theorem}
\begin{proof}
For an integer $m$ and a subset $J$ of $\{ 1, 2, ..., n\}$ having $j$ elements, let us define
\[
S =    V_n(m,J) \quad \text{ and } \quad
T  = \sum_{ J_1 \subset J} V_n(m, J_1) + \sum_{m_1 < m} V_n(m_1, J)
\]
where the first summation for $T$ is over all proper subsets $J_1$ of $J$, and the second summation 
is over all $m_1$ strictly less than $m$. 
Then, we have
\begin{align*}
S \cap T & \; = \;   V_n(m,J) \cap \left(  \sum_{ J_1 \subset J} V_n(m, J_1) + \sum_{m_1 < m} V_n(m_1, J) \right)     \\  
              & \; = \;   V_n(m,J) \cap \left(  \sum_{ J_1 \subset J} V_n(m, J_1)  \right)     \\  
              & \; = \;   \sum_{J'} V_{n}(m,J')  
\end{align*}
where the last summation is over all subsets $J'$ of $J$ having $j-1$ elements. 
Note that
\[
S+T =  L_n^{\leqslant (m, J)} \quad \text{and} \quad T = L_n^{< (m, J)}
\]
and then, by the usual module isomorphism theorem 
$(S + T) / T \cong S / (S \cap T)$, we obtain
\begin{equation} \label{LL-W-iso}
L_n^{\leqslant (m, J)} / L_n^{< (m, J)}  \cong  V_n(m,J)/  \sum_{ J'} V_n(m, J') 
\end{equation}
where the summation is over all subsets $J'$ of $J$ having $j-1$ elements.
Now, using the realization of $W_n(m,J)$ given in Lemma \ref{W-quot-0-isom},
we have 
\[
L_n^{\leqslant (m, J)} / L_n^{< (m, J)} \cong W_n(m,J).
\]
\end{proof}

In \S \ref{sec-W-alpha}, we will investigate the simple modules $W_n(m,J)$ 
in a more general setting.

\medskip


\section{Modules $L^{\boldsymbol{\alpha}}_{n}(m,j)$ 
                  and $V^{\boldsymbol{\alpha}}_n(m,J)$}


In this section, generalizing $L_n(m, J)$ and $V_n(m, J)$ discussed in the previous section,  
we define some submodules of $L_n$ over the universal enveloping algebra $\mathcal{U}_n$ of 
$\mathfrak{gl}(n)$ parameterized by 
$\boldsymbol{\alpha}=(\alpha_1, ..., \alpha_n) \in \mathbb{C}^n$.

\begin{notation}\label{notation-exp-e}
\begin{enumerate}
\item For a finite set $S$, we will write $|S|$ for the cardinality of $S$.

\item For $\boldsymbol{\alpha}=(\alpha_1, ..., \alpha_n) \in \mathbb{C}^n$, 
we write $\boldsymbol{\alpha}[\ell]$ for the $\ell$th component
$\alpha_{\ell}$ of $\boldsymbol{\alpha}$. Then, for 
$\boldsymbol{\alpha}, \boldsymbol{\beta} \in \mathbb{C}^n$, we let 
$\boldsymbol{\alpha} \pm \boldsymbol{\beta}$ be the elements in 
$\mathbb{C}^n$ such that
\[
(\boldsymbol{\alpha}  \pm \boldsymbol{\beta})[\ell]
    = \boldsymbol{\alpha}[\ell] \pm \boldsymbol{\beta}[\ell]  
    \quad \text{for $1\leq \ell \leq n$.}
\]

\item We write $\mathbf{e}_j$ for the element in $\mathbb{Z}^n$ 
whose $j$th entry is one and all the other entries are zero.
\[
\mathbf{e}_j[\ell]= 
   \left\{ \begin{array}{ll} 1 & \text{if $\ell=j$,} \\
            0 & \text{otherwise.} \end{array} \right.
\]

\item For $\mathbf{k} \in \mathbb{Z}^n$, let $\mathbf{x}^{\mathbf{k}}$ 
be the monomial 
\[
\mathbf{x}^{\mathbf{k}}
  = x_1^{\mathbf{k}[1]} x_2^{\mathbf{k}[2]} \cdots x_n^{\mathbf{k}[n]}
\]
in $L_n$.  In this setting, we denote the negative part of $\mathbf{k}$ by
\[
\mathbf{k}^{neg} = \{ \ell : \mathbf{k}[\ell] < 0 \}.
\]
\end{enumerate}
\end{notation}

We let  $E_{ab} \in \mathfrak{gl}(n)$ be
the $n \times n$ matrix with one in $(a,b)$ and zero elsewhere.
\begin{definition}\label{alpha-action-deff}
For each $\boldsymbol{\alpha} \in \mathbb{C}^n$, 
$E_{ab}$ acts on the monomials $\mathbf{x}^{\mathbf{k}}$ in the algebra 
$L_n$ of Laurent polynomials as
\[
E_{ab} \cdot \mathbf{x}^{\mathbf{k}} \; = \; 
     \left( \mathbf{k}[b] + \boldsymbol{\alpha}[b] \right) \, 
       \mathbf{x}^{\mathbf{k} + \mathbf{e}_a - \mathbf{e}_b }
     \qquad \text{for $1\leq a,b \leq n$}.
\] 
With this action, the space $L_n$ gives rise to a $\mathcal{U}_n$-module, 
which we will denote by $L_n^{\boldsymbol{\alpha}}$, and
for $f \in L_n$ we write $\langle f \rangle$ for 
the cyclic submodule of $L_n^{\boldsymbol{\alpha}}$ generated by $f$.
\end{definition}

Informally, we may think of the above action as
\[
E_{ab} \cdot f \; = \; 
  \mathbf{x}^{- \boldsymbol{\alpha}} \, x_a \frac{\partial \ }{\partial x_{b}} 
      \,  \mathbf{x}^{\boldsymbol{\alpha}} f \quad \text{for $f \in L_n$}
\] 
and then the action in the definition can be considered a generalization 
of the action \eqref{poly-action-Pn} of $\mathfrak{gl}(n)$ on the polynomial 
ring which provides all the finite dimensional 
representations of $\mathfrak{gl}(n)$ labeled by Young diagrams 
with single rows. See Theorem \ref{L-indecomp-positive} (2).

\begin{lemma}  \label{equi-iso}
For $\boldsymbol{\alpha}$ and $\boldsymbol{\beta} \in \mathbb{C}^n$, 
if $\boldsymbol{\alpha}-\boldsymbol{\beta} \in \mathbb{Z}^n$ then, 
as a $\mathcal{U}_n$-module, $L_{n}^{\boldsymbol{\alpha}}$ is 
isomorphic to $L_{n}^{\boldsymbol{\beta}}$.
\end{lemma}

\begin{proof}
It is enough to show that the linear map $\psi$ from $L_n$ to $L_n$ 
sending $\mathbf{x}^{\mathbf{k}}$ to 
$\mathbf{x}^{\mathbf{k}+ \boldsymbol{\alpha} - \boldsymbol{\beta}}$ 
for $\mathbf{k} \in \mathbb{Z}^n$ gives a $\mathcal{U}_n$-module map 
from $L_{n}^{\boldsymbol{\alpha}}$ to $L_{n}^{\boldsymbol{\beta}}$. 
It follows from
\begin{align*}
\psi ( E_{ab} \cdot \mathbf{x}^{\mathbf{k}} ) &=
\left( \mathbf{k}[b] + \boldsymbol{\alpha}[b] \right) 
       \mathbf{x}^{\mathbf{k}+ \mathbf{e}_a - \mathbf{e}_b} \times
       \mathbf{x}^{ \boldsymbol{\alpha}- \boldsymbol{\beta}} \\
 & =  \left\{ (\mathbf{k} + \boldsymbol{\alpha} - \boldsymbol{\beta})[b]
  + \boldsymbol{\beta}[b] \right\} 
   \mathbf{x}^{(\mathbf{k}+\boldsymbol{\alpha} - \boldsymbol{\beta}) 
  + \mathbf{e}_a - \mathbf{e}_b} \\
 & = E_{ab} \cdot  \mathbf{x}^{(\mathbf{k}+\boldsymbol{\alpha}
  - \boldsymbol{\beta})}  
 = E_{ab} \cdot  \psi(  \mathbf{x}^{\mathbf{k}} )         
\end{align*}
for all $1 \leq a, b \leq n$.
\end{proof}

With this lemma, we can focus on 
the following choice of $\boldsymbol{\alpha}$.

\begin{notation}\label{alpha-I}
Once and for all, we fix $\boldsymbol{\alpha}\in \mathbb{C}^n$ 
the entries of whose real parts satisfy
\[
0 \le \mathrm{Re} (\boldsymbol{\alpha}[{\ell}]) < 1 \qquad 
 \text{for all $1 \le \ell \le n$,}
\]
and the following subset of $\{1, 2, ...,n \}$
\[
I_{\boldsymbol{\alpha}}=\{ \ell : \boldsymbol{\alpha}[{\ell}]=0 \}.
\]
\end{notation}

With this choice of $\boldsymbol{\alpha}$, we note that for 
$\mathbf{k} \in \mathbb{Z}^n$,
\begin{equation}\label{integer-zero}
\mathbf{k}[\ell] + \boldsymbol{\alpha}[\ell]= 0 \quad \text{for some $\ell$}
\end{equation}
only when $\mathbf{k}[\ell]=\boldsymbol{\alpha}[\ell]=0$.

\begin{definition}[Submodules $L^{\boldsymbol{\alpha}}_{n}(m,j)$ 
                       and $V^{\boldsymbol{\alpha}}_n(m,J)$]\label{L-V-def}
\ 
\begin{enumerate}
\item For integers $m$ and $j$ with 
 $0 \leq j \leq |I_{\boldsymbol{\alpha}}|$, 
we let $L^{\boldsymbol{\alpha}}_{n}(m,j)$ be the subspace of 
$L^{\boldsymbol{\alpha}}_n$ spanned 
by all the monomials $\mathbf{x}^{\mathbf{k}}$ such that
\[
\sum_{\ell=1}^n \mathbf{k}[\ell] = m \quad \text{and} 
     \quad | \mathbf{k}^{neg} \cap I_{\boldsymbol{\alpha}} | \leq j.
\]
\item For a subset $J$ of $I_{\boldsymbol{\alpha}}$ with $|J|=j$, 
we let $V^{\boldsymbol{\alpha}}_n(m,J)$ be the subspace of 
$L^{\boldsymbol{\alpha}}_{n}(m,j)$ spanned by all the monomials 
$\mathbf{x}^{\mathbf{k}}$ in $L^{\boldsymbol{\alpha}}_{n}(m,j)$ 
such that
\[
( \mathbf{k}^{neg} \cap I_{\boldsymbol{\alpha}} ) \, \subseteq\,  J.
\]
\end{enumerate}
\end{definition}

\begin{example}
Let $n=4$, $\boldsymbol{\alpha}=(0, 1/2, 0, 0)$, and therefore 
$I_{\boldsymbol{\alpha}}=\{1,3,4\}$.
\begin{enumerate}
\item  Let $j=2$. From the condition 
$|\mathbf{k}^{neg} \cap \{1,3,4\}|\leq 2$, the space
$L^{\boldsymbol{\alpha}}_n (m,j)$ is spanned by all the monomials 
$\mathbf{x}^{\mathbf{k}}=x_1^{k_1}x_2^{k_2}x_3^{k_3}x_4^{k_4}$ 
in $L_n$ such that
\[
i) \ k_1 + k_2 + k_3 + k_4 = m \text{\ \ and \ }  
ii) \ \text{$k_1 \geq 0$ or $k_3 \geq 0$ or $k_4 \geq 0$}. 
\]
\item Let $J=\{1,4\}$. From the condition 
$( \mathbf{k}^{neg} \cap \{1,3,4\} ) \subset \{1,4\}$, 
$V^{\boldsymbol{\alpha}}_n (m,J)$ is the subspace of 
$L^{\boldsymbol{\alpha}}_n (m,j)$ spanned by all the monomials 
of degree $m$ with $k_3 \geq 0$.
\end{enumerate}
\end{example}

See more examples in Example \ref{example-2}. 
The following is easy to check.

\begin{lemma}\label{V-L-equal}
If $J=\varnothing$ (therefore $j=0$) or $J=I_{\boldsymbol{\alpha}}$ 
(therefore $j=| I_{\boldsymbol{\alpha}}|$), then 
\[
V^{\boldsymbol{\alpha}}_{n}(m,J) = L^{\boldsymbol{\alpha}}_{n}(m,j).
\]
\end{lemma}

From Definition \ref{L-V-def}, it immediately follows that
\begin{equation}\label{basic-1}
L^{\boldsymbol{\alpha}}_n (m,j-1) \subseteq L^{\boldsymbol{\alpha}}_n (m,j) 
\quad \text{and} \quad
L^{\boldsymbol{\alpha}}_n (m,j) = \sum_{J:|J|=j} V^{\boldsymbol{\alpha}}_{n} (m,J) 
\end{equation}
where the summation runs over all subsets $J$ of $I_{\boldsymbol{\alpha}}$ 
with $|J|=j$. Also, for two subsets $J_{1}$ and $J_{2}$ of 
$I_{\boldsymbol{\alpha}}$, we have
\begin{align}\label{basic-3}
& V^{\boldsymbol{\alpha}}_n(m,J_{1}) \subset V^{\boldsymbol{\alpha}}_n(m,J_{2})
\quad \text{for $J_{1}\subset J_{2}$}; \\
& V^{\boldsymbol{\alpha}}_n(m, J_1 \cap J_2) 
 = V^{\boldsymbol{\alpha}}_n(m, J_1) \cap V^{\boldsymbol{\alpha}}_n(m, J_2). \notag
\end{align}

\medskip

Now we show that $V^{\boldsymbol{\alpha}}_n(m, J)$ and 
$L^{\boldsymbol{\alpha}}_n(m,j)$ are indeed modules  
over $\mathcal{U}_n$.

\begin{proposition}\label{VL-modules}
The spaces $V^{\boldsymbol{\alpha}}_{n}(m,J)$ and $L^{\boldsymbol{\alpha}}_n(m,j)$ 
are $\mathcal{U}_n$-submodules of $L_n^{\boldsymbol{\alpha}}$.
\end{proposition}
\begin{proof}
For a monomial $\mathbf{x}^{\mathbf{p}} \in V^{\boldsymbol{\alpha}}_{n}(m,J)$,
we need to show
\[
E_{ab} \cdot \mathbf{x}^{\mathbf{p}}
 = ( \mathbf{p}[b] + \boldsymbol{\alpha}[b] ) \, 
 \mathbf{x}^{\mathbf{p}+\mathbf{e}_a-\mathbf{e}_b}
\]
are in $V^{\boldsymbol{\alpha}}_{n}(m,J)$ for all $1 \leq a,b \leq n$. 
Since the action of $E_{ab}$ preserves the degree of monomials,
writing $\mathbf{q}=({\mathbf{p}+\mathbf{e}_a-\mathbf{e}_b})$, it is 
enough to show that for $a\ne b$  if 
$(\mathbf{p}^{neg} \cap I_{\boldsymbol{\alpha}}) \subseteq J$, then 
$(\mathbf{q}^{neg} \cap I_{\boldsymbol{\alpha}}) \subseteq J$ or 
the coefficient $(\mathbf{p}[b] + \boldsymbol{\alpha}[b])$ is zero.

For this, because the only element which is not in $\mathbf{p}^{neg}$ 
but can possibly appear in $\mathbf{q}^{neg}$ is $b$, 
it is enough to consider the case
\[
b \notin \mathbf{p}^{neg} \cap I_{\boldsymbol{\alpha}} 
  \quad \text{and} \quad b \in \mathbf{q}^{neg} \cap I_{\boldsymbol{\alpha}}.
\]
This happens only when $\mathbf{p}[b]=0$ and in this case, 
since $b\in I_{\boldsymbol{\alpha}}$, we have $\boldsymbol{\alpha}[b]=0$.
Therefore, the coefficient $( \mathbf{p}[b] + \boldsymbol{\alpha}[b] )$ is 
zero. Consequently, we have 
$E_{ab} \cdot \mathbf{x}^{\mathbf{p}} \in V^{\boldsymbol{\alpha}}_{n}(m,J)$  
for all $1 \leq a,b \leq n$, and $V^{\boldsymbol{\alpha}}_{n}(m,J)$ 
is a submodule of $L^{\boldsymbol{\alpha}}_n$.
Now from \eqref{basic-1}, $L^{\boldsymbol{\alpha}}_{n}(m,j)$ is  a submodule  of 
$L^{\boldsymbol{\alpha}}_n$.
\end{proof}

\medskip


\section{Structure of $V^{\boldsymbol{\alpha}}_{n}(m,J)$}


In this section, we investigate the structure of 
$V^{\boldsymbol{\alpha}}_{n}(m,J)$. We first give a technical lemma.

\begin{lemma}\label{technical-lemma}
For two distinct monomials $\mathbf{x}^{\mathbf{p}}$ and 
$\mathbf{x}^{\mathbf{q}}$ in $L^{\boldsymbol{\alpha}}_n(m,j)$ such that 
\[
\left( \mathbf{q}^{neg}\cap I_{\boldsymbol{\alpha}} \right) 
  \; \subseteq \; 
  \left( \mathbf{p}^{neg} \cap I_{\boldsymbol{\alpha}} \right),
\]
there exists $X\in\mathcal{U}_n$ such that 
$X\cdot\mathbf{x}^{\mathbf{p}}=\mathbf{x}^{\mathbf{q}}$.
\end{lemma}
\begin{proof}
For simplicity, we let $\mathbf{p}[\ell]=p_{\ell}$, 
$\mathbf{q}[\ell]=q_{\ell}$, and 
$\boldsymbol{\alpha}[\ell]=\alpha_{\ell}$ for all $\ell$.
Consider the difference 
$\mathbf{p-q}=(p_1 - q_1 , ... , p_n - q_n ) \in \mathbb{Z}^{n}$. 
Since  $\mathbf{x}^{\mathbf{p}}$ and $\mathbf{x}^{\mathbf{q}}$ 
have the same degree $m$, we have 
$\sum_{\ell=1}^{n}{(p_{\ell} - q_{\ell})}=0$ and therefore 
we can separate the positive and negative parts of $\mathbf{p-q}$
\[
r =\sum_{\ell : p_{\ell} - q_{\ell} > 0}{(p_{\ell} - q_{\ell})}
 =\sum_{\ell : p_{\ell} - q_{\ell} < 0}{(q_{\ell} - p_{\ell})}.
\]
With Notation \ref{notation-exp-e}, we define
$1\leq s_1 \leq s_2 \leq \cdots \leq s_r \leq n$
and $1 \leq t_1 \leq t_2 \leq \cdots \leq t_r \leq n$ so that
\begin{align*}
\sum_{k=1}^{r}{\mathbf{e}_{s_k}}=\sum_{\ell:p_\ell-q_\ell>0}{(p_{\ell}-q_{\ell}) \, \mathbf{e}_\ell}
\quad \text{ and } \quad 
\sum_{k=1}^{r}{\mathbf{e}_{t_k}}=\sum_{\ell:p_\ell-q_\ell<0}{(q_{\ell}-p_{\ell}) \, \mathbf{e}_\ell}
\end{align*}
where the summations are over $\ell$ such that 
$p_{\ell}-q_{\ell} >0$ and $p_{\ell}-q_{\ell}<0$ respectively.

Setting $\mathbf{p}_0=\mathbf{p}$ and 
$\mathbf{p}_{k}=\mathbf{p}_{k-1}+\mathbf{e}_{t_{k}}-\mathbf{e}_{s_{k}}$ 
for $1\le k \le r$, we have 
$E_{t_k s_k}\cdot\mathbf{x}^{\mathbf{p}_{k-1}}= w_k \mathbf{x}^{\mathbf{p}_k}$. 
Furthermore, from  
\begin{align*}
\mathbf{p}_r
&=\mathbf{p}_0+\sum_{k=1}^{r}(\mathbf{e}_{t_k}-\mathbf{e}_{s_k})\\
&=\mathbf{p}_0+\sum_{\ell:p_\ell-q_\ell<0}{(q_{\ell}-p_{\ell})\mathbf{e}_\ell} 
  - \sum_{\ell:p_\ell-q_\ell>0}{(p_{\ell}-q_{\ell})\mathbf{e}_\ell}\\
&=\mathbf{p}_0+(\mathbf{q}-\mathbf{p})=\mathbf{q},
\end{align*}
by setting $Y=\prod_{k=1}^{r}{E_{t_k s_k}} \in \mathcal{U}_n$, we obtain
\begin{align*}
Y \cdot\mathbf{x}^{\mathbf{p}}
&=\mathbf{x}^{\mathbf{p}}(w_1 x_{t_1} x_{s_1}^{-1})
     (w_2 x_{t_2} x_{s_2}^{-1})\cdots(w_r x_{t_r} x_{s_r}^{-1})\\
&=\left(\prod_{k=1}^{r}{w_k}\right)\mathbf{x}^{\mathbf{p}}\mathbf{x}^{\mathbf{q-p}}
 =\left(\prod_{k=1}^{r}{w_k}\right)\mathbf{x}^{\mathbf{q}}
\end{align*}
where the coefficient is 
\[
\prod_{k=1}^{r}{w_k}
= \prod_{{\ell} : p_{\ell} - q_{\ell} >0} {({p}_{\ell}
 + {\alpha}_{\ell}) ({p}_{\ell} - 1 + {\alpha}_{\ell}) 
 \cdots ({p}_{\ell}  - ({p}_{\ell} - {q}_{\ell} -1) + {\alpha}_{\ell})}.
\]

For each $\ell$ in the above product, 
if ${\ell} \notin I_{\boldsymbol{\alpha}}$ then ${\alpha}_{\ell} \ne 0$, 
and therefore the corresponding factor is not zero by \eqref{integer-zero}.
Now let ${\ell} \in I_{\boldsymbol{\alpha}}$ and therefore ${\alpha}_{\ell}=0$.
To derive a contradiction, suppose $(p_{\ell}-c)$ in the $\ell$th factor of 
the above product is zero for some $0 \leq c \leq (p_{\ell}-q_{\ell}-1)$. 
Then, from $p_{\ell}=c$, we have 
\begin{equation}\label{p-range}
0 \leq p_{\ell} \leq (p_{\ell}-q_{\ell}-1)
\end{equation} 
and therefore $q_{\ell} \leq -1$. On the other hand, from the given hypothesis 
$\left( \mathbf{q}^{neg}\cap I_{\boldsymbol{\alpha}} \right) 
  \; \subseteq \; \left( \mathbf{p}^{neg} \cap I_{\boldsymbol{\alpha}} \right)$, 
we know that $p_{\ell} <0$ whenever $q_{\ell}<0$, which contradicts 
to \eqref{p-range}. Therefore, we have $\prod_{k=1}^{r}{w_k} \ne 0$ 
and with the element
\[
X= \prod_{k=1}^{r} w_k^{-1} {E_{t_k s_k}}
\]
we see that $X \cdot \mathbf{x}^{\mathbf{p}}=\mathbf{x}^{\mathbf{q}}$.
\end{proof}

\begin{example}\label{example-2}
\begin{enumerate}
\item Let $\boldsymbol{\alpha}=(1/2, i, 0)$, 
$I_{\boldsymbol{\alpha}}=\{ 3\}$, and $J=\varnothing$. 
If $m=4$ then from the condition \[
\mathbf{k}^{neg} \cap \{ 3 \} \subseteq \varnothing, 
\] 
$V^{\boldsymbol{\alpha}}_n(m,J)$ is spanned by all 
the monomials $x_1^{k_1} x_2^{k_2} x_3^{k_3}$ of degree $4$ 
with $k_3 \geq 0$.
 
Note that $\mathbf{x}^{\mathbf{p}}=x_1^4$ and 
$\mathbf{x}^{\mathbf{q}}=x_1 x_2^{-2} x_3^{5}$ in 
$V^{\boldsymbol{\alpha}}_n(m,J)$ satisfy the condition in 
Lemma \ref{technical-lemma}. From
\[
\mathbf{p}-\mathbf{q}=(3, 2, -5)=(3,2,0) + (0,0, -5)
\]
we define the element 
$Y=E_{31} \cdot E_{31} \cdot E_{31} \cdot E_{32} \cdot E_{32}
 \in \mathcal{U}_n$ to obtain
\[
Y \cdot \mathbf{x}^{\mathbf{p}}=(2 + 1/2)(3 + 1/2)(4 + 1/2)(-1 +i)(0+i) 
 \, \mathbf{x}^{\mathbf{q}}.
\]
Therefore there is $X  \in \mathcal{U}_n$ such that 
$X \cdot \mathbf{x}^{\mathbf{p}}= \mathbf{x}^{\mathbf{q}}$.

\item Let $\alpha=(0, 0, 0)$, $ I_{\boldsymbol{\alpha}}=\{1, 2, 3\}$, 
and $J=\{ 1,3\}$. If $m=-2$ then from the condition 
\[
\mathbf{k}^{neg} \cap \{1,2,3 \} \subseteq \{1,3 \},
\]
$V^{\boldsymbol{\alpha}}_n(m,J)$ is spanned by 
all the monomials $x_1^{k_1} x_2^{k_2} x_3^{k_3}$ 
of degree $-2$ with $k_2 \geq 0$. 

Note that $\mathbf{x}^{\mathbf{p}}=x_1^{-1} x_3^{-1}$ and 
$\mathbf{x}^{\mathbf{q}}=x_1 x_2^2 x_3^{-5}$ in  
$V^{\boldsymbol{\alpha}}_n(m,J)$ satisfy the condition in 
Lemma \ref{technical-lemma}. From
\[
\mathbf{p}-\mathbf{q}=(-2, -2, 4) = (0, 0, 4) + (-2,-2,0)
\] 
we define $Y=E_{13} \cdot E_{13} \cdot E_{23} \cdot E_{23} 
 \in \mathcal{U}_n$ to obtain
\[
Y \cdot \mathbf{x}^{\mathbf{p}} = (-4 + 0)(-3 +0)(-2 + 0)(-1 +0)
 \, \mathbf{x}^{\mathbf{q}}.
\]
Therefore, there is $X  \in \mathcal{U}_n$ such that 
$X \cdot \mathbf{x}^{\mathbf{p}}= \mathbf{x}^{\mathbf{q}}$.
\end{enumerate}
\end{example}

\medskip

Now we investigate the structure of $V^{\boldsymbol{\alpha}}_n(m,J)$
for $\boldsymbol{\alpha} \ne \mathbf{0}$.

\begin{theorem}
[Structure of $V^{\boldsymbol{\alpha}}_n(m,J)$ 
    with nonzero $\boldsymbol{\alpha}$]\label{V-nonzero-alpha}
\ 

Let $\boldsymbol{\alpha} \ne \mathbf{0}$ and 
$J=\{\ell_1,...,\ell_j \} \subseteq I_{\boldsymbol{\alpha}}$.
Then, $V^{\boldsymbol{\alpha}}_n(m,J)$ is the cyclic submodule of 
$L^{\boldsymbol{\alpha}}_n$ generated by
\[
\mathbf{x}_{J,t} = x_{\ell_1}^{-1} x_{\ell_2}^{-1} \cdots x_{\ell_j}^{-1} x_{t}^{m+j}
\quad \text{for some $t\in \{1,2,...,n \} \setminus I_{\boldsymbol{\alpha}}$}.
\]
\end{theorem}
\begin{proof}
Since  $\boldsymbol{\alpha} \ne \mathbf{0}$, there exists $t$ 
such that $\boldsymbol{\alpha}[t] \ne 0$. Write 
$\mathbf{x}^{\mathbf{p}}$ for $\mathbf{x}_{J,t}$ and 
let $\mathbf{x}^{\mathbf{q}}$ be an arbitrary monomial in 
$V^{\mathbf{\boldsymbol{\alpha}}}_n(m,J)$.
Since $(\mathbf{q}^{neg}\cap  I_{\boldsymbol{\alpha}})
 \subseteq (\mathbf{p}^{neg} \cap  I_{\boldsymbol{\alpha}})=J$, 
we can apply Lemma \ref{technical-lemma} to obtain 
$X \in \mathcal{U}_n$ such that $X \cdot \mathbf{x}^{\mathbf{p}}
 = \mathbf{x}^{\mathbf{q}}$. Therefore, $\mathbf{x}^{\mathbf{p}}$  
generates the module $V^{\boldsymbol{\alpha}}_n(m,J)$.
\end{proof}

\medskip

Next we consider the other cases with 
$\boldsymbol{\alpha} = \mathbf{0}$. Let us fix
\[
J=\{\ell_1,...,\ell_j \} \subseteq I_{\boldsymbol{\alpha}}=\{1,2,...,n \}.
\]

\begin{theorem}[Structure of $V^{\mathbf{0}}_n(m,J)$ 
    with nonnegative degree $m$]\label{V-cyclic-plus} 
\ 
\begin{enumerate}
\item If $m \geq 0$ and $0 \leq j \leq n-1$, then 
$V^{\mathbf{0}}_n(m,J)$ is the cyclic submodule of 
$L^{\mathbf{0}}_n(m,j)$ generated by 
\[
\mathbf{x}_{J,t} = x_{\ell_1}^{-1} x_{\ell_2}^{-1} \cdots 
 x_{\ell_j}^{-1} x_{t}^{m+j} 
 \quad \text{for some $t\in \{1,2,...,n \} \setminus J$}.
\]
\item If $m \geq 0$ and $j=n$ (therefore, $J=\{1,2,...,n \}$), 
then 
\begin{align*}
V^{\mathbf{0}}_n(m,\{1,2,...,n \}) 
 & = L^{\mathbf{0}}_n(m,n) = L^{\mathbf{0}}_n(m,n-1) \\
 & = \sum_{J'} V^{\mathbf{0}}_n(m,J')
\end{align*}
where the summation is over all the subsets $J'$ of 
$\{1,2,...,n \}$ with $|J'|=n-1$. 
\end{enumerate}
\end{theorem}
\begin{proof}
For Statement (1), we first note that 
$\mathbf{x}_{J,t} \in V^{\mathbf{0}}_n(m,J)$. Write 
$\mathbf{x}^{\mathbf{p}}$ for $\mathbf{x}_{J,t}$ and 
let $\mathbf{x}^{\mathbf{q}}$ be an arbitrary monomial 
in $V^{\mathbf{0}}_n(m,J)$. 
Since $\mathbf{q}^{neg} \subseteq \mathbf{p}^{neg}=J$, 
applying Lemma \ref{technical-lemma}, we see that there exists
$X \in \mathcal{U}_n$ such that 
$X \cdot \mathbf{x}^{\mathbf{p}} = \mathbf{x}^{\mathbf{q}}$. 
Therefore, $\mathbf{x}^{\mathbf{q}}$ belongs to the module 
generated by $\mathbf{x}_{J,t}$ and we have
$V^{\mathbf{0}}_n(m,J) = \langle \mathbf{x}_{J,t} \rangle$.

For Statement (2), if  $x_1^{k_1} x_2^{k_2} \cdots x_n^{k_n}$ is 
a monomial in $L^{\mathbf{0}}_n(m,n)$, then not all $k_j$'s can be negative, 
because the degree $m$ is nonnegative.
Therefore, $L^{\mathbf{0}}_n(m,n) = L^{\mathbf{0}}_n(m,n-1)$. 
The other equalities follow from Lemma \ref{V-L-equal} and \eqref{basic-1}.
\end{proof}

\begin{theorem}[Structure of $V^{\mathbf{0}}_n(-m,J)$ 
          with negative degree $-m$] \label{V-cyclic-minus}
\ 
 
Let $1 \leq m < n$.
\begin{enumerate}
\item[1a)] If $j=0$ (therefore, $J=\varnothing$), then 
\[
V^{\mathbf{0}}_{n}(-m,\varnothing)=L^{\mathbf{0}}_{n}(-m,0)=\{ 0\}.
\]

\item[1b)] If $1 \leq j \leq m$, then $V^{\mathbf{0}}_{n}(-m,J)$ is 
the cyclic submodule of $L^{\mathbf{0}}_{n}(-m,j)$ generated by 
\[
\mathbf{x}_J = x_{\ell_1}^{-1} x_{\ell_2}^{-1}
        \cdots x_{\ell_{j-1}}^{-1} x_{\ell_j}^{j-1-m}.
\]

\item[1c)] If $m+1 \leq j \leq n-1$, then $V^{\mathbf{0}}_n(-m,J)$ is  
the cyclic submodule of $L^{\mathbf{0}}_n(-m,j)$ generated by 
\[
\mathbf{x}_{J,t}=x_{\ell_1}^{-1} x_{\ell_2}^{-1}
        \cdots x_{\ell_{j-1}}^{-1} x_{\ell_j}^{-1} x_{t}^{j-m}
\quad \text{for some $t \in \{1,2,...,n \} \setminus J$}.
\]

\item[1d)] 
If $j=n$ (therefore, $J=\{1,2,...,n \}$), then 
\begin{align*}
V^{\mathbf{0}}_{n}(-m,\{1,2,...,n \})  
 & = L^{\mathbf{0}}_{n}(-m,n)=  L^{\mathbf{0}}_{n}(-m,n-1) \\
 & = \sum_{J'}   V^{\mathbf{0}}_{n}(-m,J') 
\end{align*}
where the summation is over all 
$J' \subset \{1,2,...,n \}$ with $|J'|= n-1$.
\end{enumerate}

Now let $m \geq n$.
\begin{enumerate}
\item[2a)] If $j=0$ (therefore, $J=\varnothing$), then 
\[
V^{\mathbf{0}}_n(-m,\varnothing)=L^{\mathbf{0}}_n(-m,0)=\{ 0\}.
\]
\item[2b)] 
If $1 \leq j  \leq n$, then $V^{\mathbf{0}}_n(-m,J)$ is 
the cyclic submodule of $L^{\mathbf{0}}_n(-m,j)$ generated by
\[
\mathbf{x}_J = x_{\ell_1}^{-1} x_{\ell_2}^{-1} \cdots 
                  x_{\ell_{j-1}}^{-1} x_{\ell_j}^{j-1-m}.
\]
\end{enumerate}
\end{theorem}

\begin{proof}
Statements 1a) and 2a) follow directly from Definition \ref{L-V-def}. 
For Statement 1d), note that if 
$\mathbf{x}^{\mathbf{k}} \in L^{\mathbf{0}}_{n}(-m,n)$ then 
$\sum_{\ell} \mathbf{k}[\ell] = -m > -n$ and therefore
there should be at least one $\ell$ with $\mathbf{k}[\ell] \geq 0$.
This implies that $L^{\mathbf{0}}_{n}(-m,n) = L^{\mathbf{0}}_{n}(-m,n-1)$. 
Now the statements follows from Lemma \ref{V-L-equal} and \eqref{basic-1}.
The other statements can be shown similarly to 
Theorem \ref{V-nonzero-alpha}  and Theorem \ref{V-cyclic-plus} (1).
\end{proof}

\begin{remark}\label{rmk-gen}
Let $\mathbf{x}^{\mathbf{p}}$ be the generators $\mathbf{x}_J$ or 
$\mathbf{x}_{J,t}$ of the cyclic modules $V^{\boldsymbol{\alpha}}_{n}(m,J)$ 
given in Theorem \ref{V-nonzero-alpha}, Theorem \ref{V-cyclic-plus}, and 
Theorem \ref{V-cyclic-minus}. We remark that these generators are 
not unique. This is because, when applying Lemma \ref{technical-lemma}, if  
\begin{equation}\label{cap-q-p-neg}
\left( \mathbf{q}^{neg}\cap I_{\boldsymbol{\alpha}} \right) 
  \; = \; \left( \mathbf{p}^{neg} \cap I_{\boldsymbol{\alpha}} \right),
\end{equation}
then we can exchange the roles of $\mathbf{x}^{\mathbf{p}}$ 
and $\mathbf{x}^{\mathbf{q}}$. Therefore, every monomial 
$\mathbf{x}^{\mathbf{q}} \in V^{\boldsymbol{\alpha}}_{n}(m,J)$ 
satisfying \eqref{cap-q-p-neg} can also generate the module 
$V^{\boldsymbol{\alpha}}_{n}(m,J)$.
\end{remark}

\medskip


\section{Structure of $L^{\boldsymbol{\alpha}}_{n}(m,j)$} \label{sec-L}


In this section, we investigate the structure of $L^{\boldsymbol{\alpha}}_n(m,j)$. 
Let us begin with another technical lemma.

\begin{lemma}\label{different-weight}
For an element $f=\sum^{r}_{i=1} c_i \mathbf{x}^{\mathbf{k}_i}$ of 
$L_{n}^{\boldsymbol{\alpha}}$ with distinct monomials and nonzero 
coefficients, the cyclic module generated by $f$ includes
the cyclic modules generated by the terms of $f$
\[
\langle \mathbf{x}^{\mathbf{k}_i} \rangle \subseteq \langle f \rangle
\quad \text{for all $i$}.
\] 
\end{lemma}

\begin{proof}
We want to prove the statement by induction on the number $r$ of the terms of $f$. 
If $r=1$ then we have nothing to prove. Suppose $r \ge 2$. We first note that, 
since $\mathbf{x}^{\mathbf{k}_i}$ are distinct, they are weight vectors with different 
weights under the action of the Cartan subalgebra $\mathfrak{h}$ of 
$\mathfrak{gl}(n)$ spanned by $E_{aa}$ for $ 1\le a \le n$.  
Let $w_i$ be the weight for the monomial $\mathbf{x}^{\mathbf{k}_i}$. 
Then, there is an element $H\in \mathfrak{h}$ such that
\[
g=w_{1}(H)f-H\cdot f=\sum_{i: w_i(H) \ne w_1(H)} (w_1(H)- w_{i}(H)) 
        \, c_{i} \mathbf{x}^{\mathbf{k}_i}
\]
is a non-zero element in $\langle f\rangle$. 
Also, since the number of terms in $g$ is less than $r$, 
by the induction hypothesis the cyclic modules generated by the terms of $g$ 
are included in $\langle g \rangle$. This shows that 
for $i$ with $w_i(H)\ne w_1(H)$ we have 
$\langle \mathbf{x}^{\mathbf{k}_i} \rangle \subseteq \langle f\rangle$. 

Now we note that
\[
h = f-\sum_{i: w_i(H)\ne w_1(H)}  c_{i} \mathbf{x}^{\mathbf{k}_i}
  =\sum_{i:w_i(H)= w_1(H)}  c_{i} \mathbf{x}^{\mathbf{k}_i}.
\]
is a nonzero element in $\langle f \rangle$.
Since the number of the terms of $h$ is less than  $r$, again 
by the induction hypothesis, the cyclic modules generated by the terms of 
$h$ are included in $\langle h \rangle$. Therefore, 
$\langle f\rangle$ contains $\langle \mathbf{x}^{\mathbf{k}_i} \rangle$ 
for $i$ with $w_i(H)= w_1(H)$ as well.
\end{proof}

As an immediate consequence of the above lemma, 
we obtain the following result for the special case of Theorem \ref{V-cyclic-minus}
with $J=\{ \ell \}$ for $1\leq \ell \leq n$. 

\begin{proposition}\label{simple-singleton}
For $m \geq 1$ and $1\leq \ell \leq n$, 
$V^{\mathbf{0}}_{n}(-m,\{ \ell \})$ is a simple submodule of 
$L^{\mathbf{0}}_{n}(-m,1)$ generated by $x_{\ell}^{-m}$.
\end{proposition}
\begin{proof}
From Theorem \ref{V-cyclic-minus} 1b) and 2b),
$V^{\mathbf{0}}_{n}(-m,J) = \langle x_{\ell}^{-m} \rangle \subseteq 
L^{\mathbf{0}}_{n}(-m,1)$.
For a nonzero $f \in V^{\mathbf{0}}_{n}(-m,J)$, writing
$f=\sum_i  c_i \mathbf{x}^{\mathbf{k}_i}$, 
let us consider the submodule of $V^{\mathbf{0}}_{n}(-m,J)$ generated by $f$. 
By Lemma \ref{different-weight}, it contains the cyclic 
submodules generated by $\mathbf{x}^{\mathbf{k}_i}$.
\[
\langle \mathbf{x}^{\mathbf{k}_i} \rangle \subseteq \langle f \rangle
 \subseteq V^{\mathbf{0}}_{n}(-m,J).
\] 

On the other hand, for each $i$, since 
$(\mathbf{k}_i^{neg} \cap \{1,2,...,n \}) \subseteq \{\ell \}$ and 
the degree of $\mathbf{x}^{\mathbf{k}_i}$ should be $-m <0$, 
we have $\mathbf{k}_{i}[\ell] < 0$ and $\mathbf{k}_{i}[\ell'] \geq 0$ 
for $\ell' \ne \ell$. By Remark \ref{rmk-gen}, 
each of these monomials
can generate the whole module $V^{\mathbf{0}}_{n}(-m,J)$. 
Therefore, we have $\langle f \rangle = V^{\mathbf{0}}_{n}(-m,J)$ 
and conclude that $V^{\mathbf{0}}_{n}(-m,\{ \ell \})$ has no nonzero 
proper submodules.
\end{proof}

\medskip

Now, we investigate the structure of $L^{\boldsymbol{\alpha}}_{n}(m,j)$  
for ${\boldsymbol{\alpha}} \ne \mathbf{0}$.
\begin{theorem}[Structure of $L^{\boldsymbol{\alpha}}_n(m,j)$ 
    with nonzero $\boldsymbol{\alpha}$]\label{L-nonzero-alpha}
    \ 
\begin{enumerate}
\item If ${\boldsymbol{\alpha}} \ne \mathbf{0}$ and 
$0 \leq j \leq |I_{\boldsymbol{\alpha}}|$, then 
$L^{\boldsymbol{\alpha}}_n(m,j)$ is indecomposable.
\item In particular, if ${\boldsymbol{\alpha}} \ne \mathbf{0}$ and $j=0$, 
then $L^{\boldsymbol{\alpha}}_n(m,0)=V^{\boldsymbol{\alpha}}_n(m, \varnothing)$ 
is a nonzero simple module over $\mathcal{U}_n$.
\end{enumerate}
\end{theorem}

\begin{proof}
For Statement (2), from Lemma \ref{V-L-equal},
$L^{\boldsymbol{\alpha}}_n(m,0)=V^{\boldsymbol{\alpha}}_n(m, \varnothing)$, and 
by Theorem \ref{V-nonzero-alpha} it is  generated by 
$\mathbf{x}_{\varnothing,t}=x_t^m$ for some 
$t \in \{1,2,...,n \} \setminus I_{\boldsymbol{\alpha}}$. 
Let $f$ be a nonzero element of $V^{\boldsymbol{\alpha}}_n(m, \varnothing)$, 
then we can write $f=\sum_{i=1}^{r} c_i \mathbf{x}^{\mathbf{k}_i}$ with 
nonzero coefficients such that 
$\mathbf{k}_{i}^{neg} \cap I_{\boldsymbol{\alpha}}=\varnothing$ for all $i$. 
By Lemma \ref{different-weight}, the monomials $\mathbf{x}^{\mathbf{k}_i}$
belong to $\langle f \rangle$. 
On the other hand, by Remark \ref{rmk-gen},
each of these monomials generates $V^{\boldsymbol{\alpha}}_n(m, \varnothing)$.
Therefore, $\langle f \rangle = V^{\boldsymbol{\alpha}}_n(m, \varnothing)$.
This shows that $L^{\boldsymbol{\alpha}}_n(m,0)
=V^{\boldsymbol{\alpha}}_n(m, \varnothing)$ is simple.

For Statement (1), we will show that every nonzero submodule $M$ of 
$L^{\boldsymbol{\alpha}}_n(m,j)$ contains 
$V^{\boldsymbol{\alpha}}_n(m, \varnothing)$, which is nonzero by Statement (2). 
For a nonzero $f \in M$, let $c \mathbf{x}^{\mathbf{p}}$ be a nonzero term of $f$.  
Then, by Lemma \ref{different-weight}, $\langle f \rangle$ includes 
$\langle \mathbf{x}^{\mathbf{p}} \rangle$.
On the other hand, since every monomial $\mathbf{x}^{\mathbf{q}}=x_t^m$ 
for $t \notin I_{\boldsymbol{\alpha}}$ satisfies the condition 
\[
\varnothing = (\mathbf{q}^{neg}\cap I_{\boldsymbol{\alpha}})
    \subseteq (\mathbf{p}^{neg}\cap I_{\boldsymbol{\alpha}}),
\] 
we can apply Lemma \ref{technical-lemma} to obtain  
$X \in \mathcal{U}_n$ such that 
$X \cdot \mathbf{x}^{\mathbf{p}} = \mathbf{x}^{\mathbf{q}}$. 
Therefore, we have
$\mathbf{x}^{\mathbf{q}}  \in \langle \mathbf{x}^{\mathbf{p}} \rangle
\subseteq \langle f \rangle \subseteq M$, and therefore  
$V^{\boldsymbol{\alpha}}_n(m, \varnothing)\subseteq M$.
This shows that  
$L^{\boldsymbol{\alpha}}_n(m,j)$ cannot be written as a direct sum 
of its proper nonzero submodules.
\end{proof}

\medskip

Next, we consider the other cases with 
$\boldsymbol{\alpha} = \mathbf{0}$.

\begin{theorem}
[Structure of  $L^{\mathbf{0}}_n(m,j)$ with nonnegative degree $m$] 
\label{L-indecomp-positive}     
\ 
\begin{enumerate}
\item If $m \ge 0$ and $1 \leq j \leq n$, then 
       $L^{\mathbf{0}}_n(m,j)$ is indecomposable.
\item In particular, if $m \ge 0$ and $j =0$, then  
   $L^{\mathbf{0}}_n(m,0)=V^{\mathbf{0}}_n(m, \varnothing)$ is 
   the cyclic module generated by 
\[
\mathbf{x}_{\varnothing,1}= x_1^m.
\]
It is a finite dimensional simple module over $\mathfrak{gl}(n)$ 
of dimension $(m+1)$.
\end{enumerate}
\end{theorem}
\begin{proof}
Statement (2) follows from the observation that for $m \geq 0$, 
$L^{\mathbf{0}}_n(m,0)$ is the space of homogeneous polynomials 
of degree $m$,
\[
L^{\mathbf{0}}_n(m,0) \cong \mathrm{Sym}^m(\mathbb{C}^n),
\]
and that $x_1^{m}$ is the highest weight vector with respect to 
the standard Borel subalgebra of upper triangular matrices in 
$\mathfrak{gl}(n)$. 

For Statement (1),
if $j=n$, then $L^{\mathbf{0}}_n(m,n)= L^{\mathbf{0}}_n(m,n-1)$
by Theorem \ref{V-cyclic-plus} (2) and therefore
we can assume $0 \leq j \leq n-1$. This case can be shown 
similarly to Theorem \ref{L-nonzero-alpha} (1). 
\end{proof}

\begin{theorem}[Structure of  $L^{\mathbf{0}}_n(-m,j)$ with 
                negative degree $-m$] \label{L-indecomp-negative}
\ 
\begin{enumerate}
\item If $m \geq 1$ and $j = 0$, then 
    $L^{\mathbf{0}}_{n}(-m,j)= V^{\mathbf{0}}_{n}(-m, \varnothing) =\{0\}$.
\item If $m \geq 1$ and $j = 1$, then $L^{\mathbf{0}}_{n}(-m,j)$ 
    decomposes into simple submodules 
\[
L^{\mathbf{0}}_{n}(-m,1)= \bigoplus_{\ell =1}^{n} V^{\mathbf{0}}_{n}(-m,\{ \ell \}).
\]
\item If $m \ge 1$ and $2 \leq j \leq n$, then $L^{\mathbf{0}}_{n}(-m,j)$ 
    is indecomposable.
\end{enumerate}
\end{theorem}
\begin{proof}
Statement (1) is straightforward to check. For Statement (2), from \eqref{basic-1} 
and Theorem \ref{V-cyclic-minus} 1b) and 2b), the module $L^{\mathbf{0}}_{n}(-m,1)$ 
is the sum of the cyclic modules 
$V^{\mathbf{0}}_{n}(-m, \{\ell \}) =\langle x_{\ell}^{-m} \rangle$,
which are simple by Proposition \ref{simple-singleton}.
Also, for $\ell \ne \ell'$, by \eqref{basic-3},
\[
V^{\mathbf{0}}_{n}(-m, \{ \ell \}) \cap  V^{\mathbf{0}}_{n}(-m, \{ \ell' \}) 
= V^{\mathbf{0}}_{n}(-m, \{\ell \}\cap \{ \ell' \})
\]
which is $ V^{\mathbf{0}}_{n}(-m, \varnothing)=\{ 0\}$ by Statement (1). 
Hence we obtain the direct sum expression.

For Statement (3), we first consider the case $2 \leq j \leq n-1$.
In order to derive a contradiction, suppose $L^{\mathbf{0}}_{n}(-m,j)=M\bigoplus N$ 
for some submodules $M$ and $N$. From \eqref{basic-1}, 
$L^{\mathbf{0}}_n(-m,j)$ is generated by the generators of 
$V^{\mathbf{0}}_n(-m,J)$ given 
in Theorem \ref{V-cyclic-minus}. We denote these generators 
$\mathbf{x}_{J}$ or $\mathbf{x}_{J,t}$ by $g_{J}$.
For each of these monomials $g_{J} \in L^{\mathbf{0}}_{n}(-m,j)$, 
if $g_J =  h_1 + h_2$ with $h_1 \in M$ and $h_2 \in N$, then $g_J$ 
appears in $h_1$ or $h_2$. Therefore, by Lemma \ref{different-weight},  
$g_J$ belongs to $\langle h_1 \rangle \subseteq M$ 
or $\langle h_2 \rangle \subseteq N$. 

If all these $g_{J}$ are in $M$ (or $N$), then 
$L^{\mathbf{0}}_n(-m,j)=M$ (or $N$). 
If some of them are in $M$ and some of them are in $N$, then 
we claim that there are $J$ and $J'$ such that $g_{J}$ is 
an element in $M$, $g_{J'}$ is 
an element in $N$, and $J \cap J' \ne \varnothing$. 
Suppose there are not such $J$ and $J'$. Then we can partition 
$\{1,2,...,n \}$ into two nontrivial parts $S_M$ and $S_N$ with 
$n_1$ elements and $n_2$ elements such that 
we have the disjoint union
\[
\{J \subset \{1,2,...,n \}:  |J|=j \}
    = \{J \subset S_M:  |J|=j \} \cup \{J \subset S_N:  |J|=j \} 
\]
where $J \subset S_M$ if and only if $g_{J}$ belongs to $M$;
$J \subset S_N$ if and only if $g_{J}$ belongs to $N$.
Note that it contradicts to $\binom{n}{j} > \binom{n_1}{j} + \binom{n_2}{j}$ 
for $2 \leq j \leq n-1$. Therefore, we conclude that there are $J$ 
and $J'$ such that $J \cap J' \ne \varnothing$. 
Now from \eqref{basic-3} we have
\[
V^{\mathbf{0}}_n(-m, J \cap J')  \subset V^{\mathbf{0}}_n(-m,J)
 \subset M \quad \text{and} \quad
 V^{\mathbf{0}}_n(-m, J \cap J')  \subset V^{\mathbf{0}}_n(-m,J')
 \subset N.
\]
Therefore, $V^{\mathbf{0}}_n(-m, J \cap J')  \ne \{ 0\}$ and $M \cap N$ 
contains a non-trivial element. 
Hence, $L^{\mathbf{0}}_n(-m,j)$ is indecomposable.

Next, let us consider the case $j=n$. 
First, if $1 \leq m <n$ then from Theorem \ref{V-cyclic-minus} 1d) 
we have $L^{\mathbf{0}}_{n}(-m,n)=  L^{\mathbf{0}}_{n}(-m,n-1)$ and 
therefore it goes back to the previous case.
Second, if $m \geq n$ then $L^{\mathbf{0}}_{n}(-m,n)
 = V^{\mathbf{0}}_{n}(-m,\{1,2,...,n \})$ 
by Lemma \ref{V-L-equal} and its generator is 
$\mathbf{x}_{\{1,2,...,n \}}=x_1^{-1} \cdots x_{n-1}^{-1} x_n^{n-1-m}$ 
by Theorem \ref{V-cyclic-minus} 2b). 
If $L^{\mathbf{0}}_{n}(-m,n)=M \oplus N$, then 
$\mathbf{x}_{\{1,2,...,n \}}= h_1 + h_2$  for some $h_1 \in M$ and 
$h_2 \in N$, and the monomial $\mathbf{x}_{\{1,2,...,n \}}$ appears 
in $h_1$ or $h_2$. By Lemma \ref{different-weight},
 $\langle \mathbf{x}_{\{1,2,...,n \}} \rangle \subseteq
  \langle h_1 \rangle \subseteq M$ or 
 $\langle \mathbf{x}_{\{1,2,...,n \}} \rangle \subseteq
  \langle h_2 \rangle \subseteq N$. 
This shows that  $M$ or $N$ should be equal to $L^{\mathbf{0}}_{n}(-m,n)$.
Therefore, $L^{\mathbf{0}}_n(-m,j)$ is indecomposable.
\end{proof}

\medskip


\section{Simple modules $W_n^{\boldsymbol{\alpha}}(m,J)$} \label{sec-W-alpha}


In this section, we investigate some submodules of the quotients 
$L^{\boldsymbol{\alpha}}_n(m, j)/L^{\boldsymbol{\alpha}}_n(m,j-1)$.
We will assume $L^{\boldsymbol{\alpha}}_{n}(m,j)=\{0\}$ for $j \leq -1$.

\begin{definition}\label{def-W}
For $m \in \mathbb{Z}$ and a subset $J$ of $I_{\boldsymbol{\alpha}}$ 
with cardinality $j$, we define the following submodule
of the quotient 
$L^{\boldsymbol{\alpha}}_{n}(m,j)/L^{\boldsymbol{\alpha}}_{n}(m,j-1)$
\[
W^{\boldsymbol{\alpha}}_{n}(m,J) 
= \left( V^{\boldsymbol{\alpha}}_{n}(m,J) 
  +L^{\boldsymbol{\alpha}}_{n}(m,j-1) \right) / L^{\boldsymbol{\alpha}}_{n}(m,j-1).
\]
\end{definition}
We note that 
\[
W^{\boldsymbol{\alpha}}_{n}(m,J)
 \cong V^{\boldsymbol{\alpha}}_{n}(m,J) / \left(V^{\boldsymbol{\alpha}}_{n}(m,J)
 \cap L^{\boldsymbol{\alpha}}_{n}(m,j-1) \right)
\]
and    
\[
V^{\boldsymbol{\alpha}}_{n}(m,J) \cap L^{\boldsymbol{\alpha}}_{n}(m,j-1)
= \sum_{J'} V^{\boldsymbol{\alpha}}_{n}(m,J')
\]
where the summation runs over all $J' \subset J$ with $|J'|=j-1$.

\bigskip

In \S \ref{sec-L}, we saw  that nontrivial modules 
$L^{\boldsymbol{\alpha}}_n(m,0)/L^{\boldsymbol{\alpha}}_n(m,-1) \cong 
L^{\boldsymbol{\alpha}}_n(m,0) =V^{\boldsymbol{\alpha}}_n(m,\varnothing)$ 
are simple, and that  
for $m \geq 1$ the following quotient decomposes into simple submodules: 
\[
L^{\mathbf{0}}_{n}(-m,1) / L^{\mathbf{0}}_{n}(-m,0) 
  \cong \bigoplus_{\ell =1}^{n} V^{\mathbf{0}}_{n}(-m,\{ \ell \}).
\]
Let us generalize these observations.

\begin{theorem}\label{W-subquot} Let $m \in \mathbb{Z}$.  
\begin{enumerate}
\item For $J \subset I_{\boldsymbol{\alpha}}$, 
the module $W^{\boldsymbol{\alpha}}_n(m,J)$ is simple.
\item For $1 \leq j \leq |I_{\boldsymbol{\alpha}}|$, the quotient module 
$L^{\boldsymbol{\alpha}}_n(m,j) /\, L^{\boldsymbol{\alpha}}_n(m,j-1)$ 
decomposes as 
\[
L^{\boldsymbol{\alpha}}_n(m,j) /\, L^{\boldsymbol{\alpha}}_n(m,j-1)
  = \bigoplus_{J: |J|=j} W^{\boldsymbol{\alpha}}_n(m,J)
\]
where the direct sum is taken over all subsets $J$ of 
$I_{\boldsymbol{\alpha}}$ with cardinality $j$.
\end{enumerate}
\end{theorem}
\begin{proof}
For Statement (1), for any nonzero element 
$\bar{f} \in W^{\boldsymbol{\alpha}}_n(m,J)$, 
we want to show that $\langle \bar{f} \rangle=W^{\boldsymbol{\alpha}}_{n}(m,J)$.
From the definition of $W^{\boldsymbol{\alpha}}_{n}(m,J)$, we can assume that   
\[
\bar{f}= f + L^{\boldsymbol{\alpha}}_{n}(m,j-1)
\] 
where $f = \sum_{i=1}^r c_{i} \mathbf{x}^{\mathbf{k}_{i}}
 \in V^{\boldsymbol{\alpha}}_n(m,J)$
having distinct monomials $\mathbf{x}^{\mathbf{k}_{i}}$
 in $V^{\boldsymbol{\alpha}}_n(m,J)$ 
with $\mathbf{k}_{i}^{neg} \cap I_{\boldsymbol{\alpha}}=J$.
From Lemma \ref{different-weight}, $\langle f \rangle$ includes 
the cyclic modules $\langle \mathbf{x}^{\mathbf{k}_{i}} \rangle$. 
On the other hand, by Theorem \ref{V-nonzero-alpha}, Theorem \ref{V-cyclic-plus}, 
Theorem \ref{V-cyclic-minus}, and Remark \ref{rmk-gen}, 
each $\mathbf{x}^{\mathbf{k}_{i}}$ generates 
the module $V^{\boldsymbol{\alpha}}_n(m,J)$. This shows 
that $\langle \bar{f} \rangle = W^{\boldsymbol{\alpha}}_{n}(m,J)$. 

For Statement (2), with $\eqref{basic-1}$ we see that
\begin{align*}
L^{\boldsymbol{\alpha}}_n(m,j) / L^{\boldsymbol{\alpha}}_n(m,j-1) 
&= \left(\sum_{J:|J|=j} V^{\boldsymbol{\alpha}}_{n} (m,J) \right) / L^{\boldsymbol{\alpha}}_n(m,j-1) \\
&=\left(\sum_{J:|J|=j} V^{\boldsymbol{\alpha}}_{n} (m,J)
+L^{\boldsymbol{\alpha}}_n(m,j-1)\right) / L^{\boldsymbol{\alpha}}_n(m,j-1). 
\end{align*}
Therefore, we have 
\[
L^{\boldsymbol{\alpha}}_n(m,j) / L^{\boldsymbol{\alpha}}_n(m,j-1) 
= \sum_{J:|J|=j} W^{\boldsymbol{\alpha}}_n(m,J)
\]
where the summation is over $J \subset I_{\boldsymbol{\alpha}}$ with $|J|=j$.
Now, suppose we have
\[
\bar{f} \in W^{\boldsymbol{\alpha}}_n(m,J_{1})
 \cap W^{\boldsymbol{\alpha}}_n(m,J_{2})
\]
with distinct subsets $J_{1}$ and $J_{2}$ of $I_{\boldsymbol{\alpha}}$. 
Then, we can assume that $\bar{f}= f + L^{\boldsymbol{\alpha}}_{n}(m,j-1)$ 
where $f = \sum_{i=1}^r c_{i} \mathbf{x}^{\mathbf{k}_{i}}$ 
with distinct monomials $\mathbf{x}^{\mathbf{k}_{i}}$ such that 
$\mathbf{k}_{i}^{neg} \cap I_{\boldsymbol{\alpha}} \subseteq J_1 \cap J_2$ 
for all $i$. Since $|J_1 \cap J_2| < j$, this shows that 
$f \in L^{\boldsymbol{\alpha}}_n(m,j-1)$ and therefore 
$\bar{f}$ is zero in the quotient 
$L^{\boldsymbol{\alpha}}_n(m,j) / L^{\boldsymbol{\alpha}}_n(m,j-1)$.
Therefore, we obtain the direct sum expression in the statement.
\end{proof}

\medskip

Next we investigate the cases when 
$W^{\boldsymbol{\alpha}}_{n}(m,J)$ are highest weight modules.

\begin{theorem}
[Highest weight vector in $W^{\boldsymbol{\alpha}}_{n}(m,J)$] \label{HW-mod}
\ 
\begin{enumerate}
\item
For an integer $1\leq \ell \leq n$, if 
$\boldsymbol{\alpha}\in \mathbb{C}^{n}$ is such that $\boldsymbol{\alpha}[\ell]=c$ 
is nonzero and $\boldsymbol{\alpha}[\ell']=0$ for all $\ell' \ne \ell$ and 
$J=\{1, 2, ..., \ell-1\} \subseteq I_{\boldsymbol{\alpha}}$, then 
for every $m \in \mathbb{Z}$ the  module 
$W^{\boldsymbol{\alpha}}_{n}(m,J)$ is a highest 
weight module having a highest weight vector
\[
(x_{1}^{-1} x_{2}^{-1} \cdots x_{\ell-1}^{-1} x_{\ell}^{m+\ell-1})
    + L^{\boldsymbol{\alpha}}_{n}(m,\ell-2)
\]
with highest weight
\[
(-1,-1, ...,-1,m+\ell-1+c,0, ...,0).
\]

\item
Let $\boldsymbol{\alpha}=\mathbf{0}$ and therefore $I_{\boldsymbol{\alpha}}=\{ 1,2,...,n\}$. 
For $1\leq \ell \leq n$, if $J=\{1,2, ... ,\ell-1\}$ then  
for $m \in \mathbb{Z}$ such that $m+\ell-1 \geq 0$ the  module 
$W^{\mathbf{0}}_{n}(m,J)$ is a highest weight module having a highest weight vector
\[
(x_{1}^{-1} x_{2}^{-1}\cdots x_{\ell-1}^{-1} x_{\ell}^{m+\ell-1})
  + L^{\mathbf{0}}_{n}(m,\ell-2)
\]
with highest weight
\[
(-1,-1, ... ,-1,m+\ell-1,0, ... ,0).
\]

In particular, if $\ell=1$ and $J=\varnothing$ then for $m \geq 0$, the module
$W^{\mathbf{0}}_{n}(m,\varnothing)$ is a $(m+1)$-dimensional module 
with highest weight $(m, 0, ..., 0)$.

\item
Let $\boldsymbol{\alpha}=\mathbf{0}$ and therefore $I_{\boldsymbol{\alpha}}=\{ 1,2,...,n\}$. 
For  $1\leq \ell \leq n$, if $J=\{1,2,\cdots,\ell\}$ then  
for $m \in \mathbb{Z}$ such that $m+\ell-1 <0$ the module $W^{\mathbf{0}}_{n}(m,J)$ 
is a highest weight module having a highest weight vector
\[
(x_{1}^{-1} x_{2}^{-1}\cdots x_{\ell-1}^{-1} x_{\ell}^{m+\ell-1} )
 + L^{\mathbf{0}}_{n}(m,\ell-1)
\]
with highest weight
\[
(-1,-1, ...,-1,m+\ell-1,0, ...,0).
\]
In particular, if $\ell=n$ and $J=\{1, 2,...,n\}$, then for $m \leq -n$
the module $W^{\mathbf{0}}_{n}(m,\{1, 2, ...,n\})$ is a finite dimensional 
module with highest weight 
\[(-1, -1,..., -1, m+n-1).\]
\end{enumerate}
\end{theorem}

\begin{proof}
We first notice that the given elements $\mathbf{x}^{\mathbf{k}}+L^{\boldsymbol{\alpha}}_{n}(m,j-1)$ 
generate $W^{\boldsymbol{\alpha}}_{n}(m,J)$ where $j=|J|$ 
(see Theorem \ref{V-nonzero-alpha}, Theorem \ref{V-cyclic-plus}, and Theorem \ref{V-cyclic-minus}).
It is straightforward to verify their weights under the action of the Cartan subalgebra of 
$\mathfrak{gl}(n)$ generated by $E_{aa}$ for $1 \leq a \leq n$. 
Therefore, now it is enough to show that 
\begin{align}\label{ann-hwv}
E_{ab} \cdot (\mathbf{x}^{\mathbf{k}} + L^{\boldsymbol{\alpha}}_{n}(m,j-1)) 
& \,= \, (\mathbf{k}[b]+\boldsymbol{\alpha}[b])  (x_a x_b^{-1}) \mathbf{x}^{\mathbf{k}}
      + L^{\boldsymbol{\alpha}}_{n}(m,j-1)  {\notag} \\
& \, =\,  L^{\boldsymbol{\alpha}}_{n}(m,j-1) 
\end{align}
in $W^{\boldsymbol{\alpha}}_{n}(m,J)$ for all $1\le a <b \le n$.

For Statement (1), if $a<b$ and $b\ge \ell+1$, then since 
$\mathbf{k}[b]=\boldsymbol{\alpha}[b]=0$ we have
\[
E_{ab} \cdot (x_{1}^{-1}\cdots x_{\ell-1}^{-1} x_{\ell}^{m+\ell-1}) \, = \,
 (0+0) (x^{1}_a x_b^{-1}) (x_{1}^{-1} \cdots x_{\ell -1}^{-1} x_{\ell}^{m+\ell-1}  ) \, = \, 0. 
\]
If $a<b$ and $b \leq  \ell$, then $a\le\ell-1$ and
\[
E_{ab} \cdot (x_{1}^{-1}\cdots x_{\ell-1}^{-1} x_{\ell}^{m+\ell-1}) \, = \, 
(\mathbf{k}[b]+\boldsymbol{\alpha}[b]) \, (x^{1}_a x_b^{-1}) ( x_{1}^{-1}   \cdots  x_{\ell -1}^{-1} x_{\ell}^{m+\ell-1} )                         
\]
where $\mathbf{k}[b]=-1$ and $\boldsymbol{\alpha}[b]=0$ if $b \leq \ell-1$; and 
$\mathbf{k}[b]= m + \ell -1$ and $\boldsymbol{\alpha}[b]=c$ if $b=\ell$. 
Writing $\mathbf{x}^{\mathbf{q}}$ for the monomial in the right hand side,
we see that  $\mathbf{q}[a]=0$ because $a \leq \ell -1$ and 
therefore $|\mathbf{q}^{neg} \cap I_{\boldsymbol{\alpha}}| < \ell -1$.
This shows that $\mathbf{x}^{\mathbf{q}} \in L^{\boldsymbol{\alpha}}_{n}(m,\ell-2)$ 
and therefore \eqref{ann-hwv} is true. 

For Statement (2), the first part can be shown similarly to the previous 
case. The second part with the conditions $\ell=1$ and $J=\varnothing$ 
follows directly from Definition \ref{def-W} with 
$L^{\mathbf{0}}_{n}(m,-1) = \{ 0\}$ and Theorem \ref{L-indecomp-positive} (2).

For Statement (3), if $a<b$ and $b\ge \ell+1$, then 
since $\mathbf{k}[b]=\boldsymbol{\alpha}[b]=0$ we have
\[
 E_{ab} \cdot (x_{1}^{-1}\cdots x_{\ell-1}^{-1} x_{\ell}^{m+\ell-1})  \,= \, 
 (0 + 0) (x^{1}_a x_b^{-1}) (x_{1}^{-1} \cdots  x_{\ell-1}^{-1} x_{\ell}^{m+\ell-1} x_b^{-1}) \,=\, 0. 
\]
If $a<b$ and $b \leq \ell$, then since $\boldsymbol{\alpha}[b]=0$ we have
\[
E_{ab} \cdot (x_{1}^{-1}\cdots x_{\ell-1}^{-1} x_{\ell}^{m+\ell-1}) 
 \, = \, \mathbf{k}[b] (x^{1}_a x_b^{-1}) ( x_{1}^{-1} \cdots  x_{\ell -1}^{-1} x_{\ell}^{m+\ell-1})
\]
where $\mathbf{k}[b]=-1$ if $b \leq \ell -1$ and  $\mathbf{k}[b]= m + \ell -1$ if $b=\ell$. 
Again, by denoting the monomial in the right hand side by $\mathbf{x}^{\mathbf{q}}$, 
 we see that  $\mathbf{q}[a]=0$ because $a\le\ell-1$, and therefore $|\mathbf{q}^{neg} \cap I_{\boldsymbol{\alpha}}| < \ell$.
This shows that $\mathbf{x}^{\mathbf{q}} \in L^{\boldsymbol{\alpha}}_{n}(m,\ell-1)$ and 
therefore \eqref{ann-hwv} is true.
\end{proof}

We note that the highest weights of $W_n^{\boldsymbol{\alpha}}(m,J)$ given in 
Theorem \ref{HW-mod} are integral dominant (see, for example, \cite[\S 3]{GW09}) only when 
\begin{enumerate}
\item[i)] $\boldsymbol{\alpha}=\mathbf{0}$, $J=\varnothing$, and $m \geq 0$;
\item[ii)] $\boldsymbol{\alpha}=\mathbf{0}$, $J=\{ 1,2,...,n\}$, and $m \leq -n$. 
\end{enumerate}
Indeed, one can easily check that these are the only cases when the modules 
$W_n^{\boldsymbol{\alpha}}(m,J)$ are finite dimensional.

\medskip


\bigskip

\end{document}